\documentclass{amsart}
\usepackage{amsfonts}
\usepackage{amsmath,amssymb}
\usepackage{amsthm}
\usepackage{amscd}
\usepackage{graphics}
\usepackage{graphicx}

\theoremstyle{remark}{
\newtheorem{Def}{{\rm Definition}}

\newtheorem{Rem}{{\rm Remark}}

}
\theoremstyle{plain}
{

\newtheorem{Thm}{Theorem}
\newtheorem{MainThm}{Main Theorem}

}

\begin{document}
\title[Algebraic realization of certain round fold maps]{On real algebraic realization of round fold maps of codimension $-1$}
\author{Naoki kitazawa}
\keywords{(Non-singular) real algebraic manifolds and real algebraic maps. Smooth maps. Fold maps. Round fold maps. Special generic maps. Graphs. Digraphs. Reeb (di)graphs. Poincar\'e-Reeb graphs. \\
\indent {\it \textup{2020} Mathematics Subject Classification}: 05C10, 14P05, 14P25, 57R45, 58C05, 58C06.}

\address{Osaka Central Advanced Mathematical Institute (OCAMI) \\
3-3-138 Sugimoto, Sumiyoshi-ku Osaka 558-8585
TEL: +81-6-6605-3103
}
\email{naokikitazawa.formath@gmail.com}
\urladdr{https://naokikitazawa.github.io/NaokiKitazawa.html}
\maketitle
\begin{abstract}
The canonical projections of the unit spheres are generalized to {\it special generic} maps and {\it round fold} maps, for example.
They are generalizations from the viewpoint of singularity theory of differentiable maps and these maps restrict the topologies and the differentiable structures of the manifolds.

We are concerned with round fold maps, defined as smooth maps locally represented as the product map of a Morse function and the identity map on a smooth manifold, and maps with singular value sets being concentric spheres. A bit different from differential topology, we are concerned with real algebraic geometric aspects of these maps. We discuss real algebraic realization of round fold maps of codimension $-1$ as our new work.
Real algebraic realization of these maps is of fundamental and important studies in real algebraic geometry and a new study recently developing mainly due to the author. 
\end{abstract}
%【REVISE】 combinatoric ～ is → combinatorial object. It is .
%【REVISE】  such that a point is a vertex if and only if the corresponding connected component of the level set contains some singular points → whose vertex set is the set of all points containing some singular points in the corresponding connected component of the level set .
%【REVISE】 We delete "extending the result before".
\section{Introduction.}
\label{sec:1}
Let ${\mathbb{R}}^k$ denote the $k$-dimensional {\it Euclidean space} ({\it real affine space}), with $\mathbb{R}:={\mathbb{R}}^1$. Let ${\pi}_{k_1+k_2,k_2}:{\mathbb{R}}^{k_1+k_2} \rightarrow {\mathbb{R}}^{k_2}$ denote the canonical projection mapping $x:=(x_1,\cdots x_{k_1},\cdots x_{k_1+k_2})$ to $(x_1,\cdots x_{k_2})$ with $k_1$ and $k_2$ being positive integers. The {\it canonical projection of the $m$-dimensional unit sphere $S^m:=\{(x_1,\cdots x_{m+1}) \mid {\Sigma}_{j=1}^{m+1} {x_j}^2=1\}$} to ${\mathbb{R}}^n$ is of most fundamental smooth (real algebraic) maps on closed manifolds (real algebraic compact manifolds), defined as the restriction of ${\pi}_{m+1,n}$.

\subsection{The canonical projection of the unit sphere and fold (special generic) maps.}
Let $c {\mid}_Z$ denote the restriction of a map $c:X \rightarrow Y$ to a subset $Z \subset X$. 
Hereafter, we also need fundamental singularity theory of differentiable maps and see \cite{golubitskyguillemin} for this.

The canonical projection ${\pi}_{m+1,n} {\mid}_{S^m}$ of $S^m$ is a simplest example of {\it special generic} maps and {\it round fold} maps. These classes of smooth maps are both subclasses of the class of so-called {\it fold} maps.
They are important in global singularity theory and applications to algebraic topology, differential topology, and geometric topology of manifolds. This is an application of higher dimensional versions of theory of Morse functions.
We explain these maps. A {\it singular} point of a smooth map $f$ between smooth manifolds is a point of the manifold of the domain where the rank of the differential is smaller than the minimum between the dimensions of the manifolds of the domain and the target. The {\it singular set} $S(f)$ of the smooth map is the set of all singular points of the map and the image is the {\it singular value set} of the map. 
We also use "critical" instead of "singular" in the case $f$ is a real-valued function. A {\it diffeomorphism} is a smooth map being a homeomorphism and having no singular point. Two manifolds $X_1$ and $X_2$ are {\it diffeomorphic}, or equivalently, a manifold is {\it diffeomorphic} to the other, if a diffeomorphism between them exists. 

A {\it fold} map is a smooth map $f:M \rightarrow {\mathbb{R}}^n$ on a manifold with no boundary of dimension $m \geq n$ whose critical point is always of the form $f(x_1,\cdots x_m)=(x_1,\cdots x_{n-1}, {\Sigma}_{j=1}^{m-n-i(p)+1} {x_{n-1+j}}^2-{\Sigma}_{j=1}^{i(p)} {x_{m-i(p)+j}}^2)$ for some local coordinates and an integer $0 \leq i(p) \leq \frac{m-n+1}{2}$. It is a kind of exercises on singularity theory of smooth maps to prove that $i(p)$ is uniquely chosen ($i(p)$ is the {\it index of $p$ for $f$}). The singular set $S(f)$ and the set $F_i(f)$ of all singular points of index $i$ for $f$ is a smooth submanifold of dimension $n-1$ with no boundary and the restriction of $f$ to $S(f)$ is a smooth immersion. A Morse function on a smooth manifold with no boundary is a fold map and for the case $n=1$ and in our paper, we also consider Morse functions on smooth manifolds with non-empty boundaries which have no critical point on the boundaries and which are constant on connected components of the boundaries. As is well-known, for Morse functions on smooth manifolds, the integer $i(p)$ can be chosen as a unique integer $0 \leq i(p) \leq m$ respecting the natural order on $\mathbb{R}$. For Morse functions, see \cite{milnor}. The paper \cite{thom, whitney} are related pioneering studies on fold maps and more general generic smooth maps into ${\mathbb{R}^2}$ and \cite{eliashberg1, eliashberg2} are on existence theory of fold maps via methods on differential equations and homotopy. The paper \cite{saeki1} is of pioneering studies on fold maps related to differential topology of manifolds.  

A {\it special generic} map is a fold map with $i(p)=0$ for each singular point $p$ of it.  ${\pi}_{m+1,n} {\mid}_{S^m}$ is a special generic map. A Morse function with exactly two critical points on a closed manifold is special generic. This is known as a smooth function on Reeb's sphere theorem. We have a natural special generic map $f:M \rightarrow {\mathbb{R}}^n$ on a manifold $M$ represented as a connected sum ${\sharp}_{j=1}^l (S^{k_j} \times S^{m-k_j})$ with $l>0$, $m \geq n \geq 2$ and $1 \leq k_j \leq n-1$ into ${\mathbb{R}}^n$ with $f {\mid}_{S(f)}$ being an embedding and the image diffeomorphic to a boundary connected sum ${\natural}_{j=1}^l (S^{k_j} \times D^{n-k_j})$ : hereafter, we consider connected sums and boundary connected sums in the smooth category unless otherwise stated and we use such notation. It is also shown that these maps restrict the topologies and the differentiable structures of the manifolds. Papers such as \cite{burletderham, furuyaporto, saeki2} are of important pioneering studies and \cite{saekisakuma1, saekisakuma2} are also important.

It is also pointed out first in \cite{kitazawa5} that a special generic map $f:M \rightarrow {\mathbb{R}}^n$ on a closed and connected manifold of a certain class including ${\pi}_{m+1,n} {\mid} S^m$ is explicitly represented as the restriction of ${\pi}_{m+1,n}$ to a real algebraic manifold. This is remarked in \cite{kitazawa7} again and discussed in a certain generalized situation. This method is important and reviewed in Theorem \ref{thm:5} in a self-contained way.
\subsection{Round fold maps and our main result.}
We state our main theorem. Before this, we review or introduce several notions.

Two smooth maps $f_i:M_i \rightarrow N_i$ with $i=1,2$ are said to be {\it $C^{\infty}$ equivalent} or $\mathcal{A}$-equivalent if and only if there exists a pair $({\phi}_M:M_1 \rightarrow M_2,{\phi}_N:N_1 \rightarrow N_2)$ with the relation ${\phi}_N \circ f_1=f_2 \circ {\phi}_M$. If we can choose ${\phi}_N$ as the identity map on $N:=N_1=N_2$, then they are also said to be {\it $\mathcal{R}$-equivalent}. These are of course equivalence relations on the class of smooth maps. We also use the phrases "$\mathcal{E}$-equivalent to" and "up to $\mathcal{E}$-equivalence", for example (for $\mathcal{E}=\mathcal{A}, \mathcal{R}$).

A {\it round} fold map $f:M \rightarrow {\mathbb{R}}^n$ on an $m$-dimensional closed and connected manifold $M$ with $m \geq n \geq 2$ is defined as a fold map $\mathcal{A}$-equivalent to a fold map $f_0:M_0 \rightarrow {\mathbb{R}}^n$ with $f_0 {\mid}_{S(f_0)}$ being an embedding and $f_0(S(f_0))={\bigcup}_{i=1}^l \{x \in {\mathbb{R}}^n \mid {\Sigma}_{j=1}^n {x_j}^2=i^2\}$ with an integer $l>0$.
% (we call this map $f$ a {\it normal form} of a round fold map and we essentially consider such maps).
For a ray $\{(tx_{0,1},\cdots tx_{0,n}) \mid t \geq \frac{1}{2}\}$ with $x_0=(x_{0,1},\cdots x_{0,n}) \in S^{n-1}$, we can define a Morse function mapping each point to $t \geq \frac{1}{2}$ as the restriction to the preimage of the ray and this is a {\it page function} ${\tilde{f}}_{{\rm P},x_0}$ of the map $f$. 
 ${\pi}_{m+1,n} {\mid}_{S^m}$ is a round fold map. We focus on these maps in the present paper. Round fold maps have been first introduced in \cite{kitazawa1, kitazawa2, kitazawa3} and studied in \cite{kitazawa4} and later in \cite{kitazawasaeki}, for example. 

On a manifold $M$ represented as a connected sum ${\sharp}_{j=1}^l (S^{n} \times S^{m-n})$ with $l>0$ and $m>n \geq 2$, there exists a round fold map $f:M \rightarrow {\mathbb{R}}^n$ and we can generalize the connected summands of $M$ to general smooth bundles over $S^{n}$ whose fibers are $S^{m-n}$: a {\it smooth} bundle is a bundle whose fiber is a smooth manifold and whose structure group consists of diffeomorphisms on the fiber and topologized with the so-called {\it Whitney $C^{\infty}$ topology}. We can have another round fold map into ${\mathbb{R}}^{n^{\prime}}$ with $n>n^{\prime} \geq 2$ by composing a round fold map into ${\mathbb{R}}^n$ with ${\pi}_{n,n^{\prime}}$.

 We can consider the quotient space of a topological space $X$ for a (continuous) map $c:X \rightarrow Y$ as follows. Two points $x_1 \in X$ and $x_2 \in X$ are equivalent and $x_1 {\sim}_c x_2$ if and only if $x_1$ and $x_2$ are in a same connected component of a same preimage $c^{-1}(y)$. We can define the quotient map $q_c:X \rightarrow R_c$ onto the quotient space $R_c:=X/{\sim c}$ with the uniquely defined continuous map $\bar{c}$ satisfying $c=\bar{c} \circ q_c$. In the case $Y:=\mathbb{R}$, each preimage $c^{-1}(y)$ is a {\it level set} of $c$ and its connected component is a {\it contour} of $c$. If $c:X \rightarrow \mathbb{R}$ is a differentiable function, then its contour having (no) critical points is a {\it critical} (resp. {\it regular}) contour. We use the notation $f:M \rightarrow \mathbb{R}$ again. The quotient space $R_f:=M/{\sim f}$ has the structure of a graph by defining its vertex suitably, in several situations. For example, consider the case of a Morse(-Bott) function $f$ on a compact manifold $M$, which has no critical point on the boundary, which is constant on each connected component of the boundary, and whose critical value set is finite. In these cases, the quotient space is regarded as a graph whose vertex $v$ is defined as a point with ${q_f}^{-1}(v)$ being a critical contour of $f$ or a connected component of the boundary $\partial M$ of the manifold $M$. The graph $R_f$ is the {\it Reeb graph} of $f$. Such objects are classical and have been strong tools in understanding manifolds by nice functions such as Morse functions (\cite{reeb}). For a Morse function on a closed manifold, $R_f$ has been explicitly shown to be a graph by this rule, in \cite{izar}. In \cite{saeki2} (\cite[Theorem 3.1]{saeki2}), the space $R_f$ has been shown to be a graph by this rule, in the case where the critical value set of a smooth function $f$ on a compact manifold is finite.

\setcounter{MainThm}{-1}
\begin{MainThm}
\label{mthm:0}
Let $m>3$ be an integer. A round fold map $f:M \rightarrow {\mathbb{R}}^{m-1}$ on an $m$-dimensional closed and connected manifold $M$ such that for a page function $\tilde{f_{{\rm P},x_0}}$, the function $\bar{\tilde{f_{{\rm P},x_0}}}$ is the composition of some embedding into ${\mathbb{R}}^2$ with ${\pi}_{2,1}$ and that either of the following is satisfied is represented as the restriction of ${\pi}_{m+1,m-1}$ to the zero set of a suitable real polynomial function, up to $\mathcal{A}$-equivalence.
\begin{enumerate}
\item \label{mthm:0.1} For a closed, connected and orientable surface $S$, $M=S^{m-2} \times S$. In addition, either $m>4$, or $m=4$ and $S(f)-F_0(f)$ is not empty.
\item \label{mthm:0.2} $M$ is orientable and $f(M)$ is diffeomorphic to $D^{m-1}$.
\end{enumerate}
\end{MainThm}
The next section is for preliminaries and we explain some terminologies and notions more precisely. In the third section, we introduce Main Theorem \ref{mthm:0} in a refined way, again, as Main Theorems \ref{mthm:1} and \ref{mthm:2}. We also prove the theorems.
\section{Preliminaries.}
We review important notions and arguments.
\subsection{Reeb (di)graphs of Morse functions.}
Hereafter, a {\it graph} means a $1$-dimensional finite and connected CW complex the closure of whose $1$-cell is always homeomorphic to $D^1$. A {\it vertex} (an {\it edge}) of a graph means a $0$-cell (resp. $1$-cell) of the graph.
The {\it degree} of a vertex of a graph means the number of its edges incident to the vertex. 
A {\it digraph} $(G,c_G)$ means a graph equipped with a continuous function $c_G:G \rightarrow \mathbb{R}$ which is injective on each edge. In this situation, each edge of $G$ is oriented by the rule that an edge $e_{v_1,v_2}$ of $G$ incident to two distinct vertices $v_1$ and $v_2$ of the graph is oriented from $v_1$ to $v_2$ if and only if $c_G(v_1)<c_G(v_2)$. A {\it source} ({\it sink}) of a digraph means its vertex from which all of its edges depart (resp. which all of its edges enter). The Reeb graph $R_f$ is regarded as the digraph $(R_f,\bar{f})$.
For two digraphs, the notion of isomorphism is defined by defining {\it isomorphisms} as isomorphisms of the (underlying) graphs preserving the orientations and the orders of the values at the vertices.
\subsection{M-digraphs.} 
A {\it Morse} digraph or an {\it M-digraph} means a digraph whose sinks and edges are all of degree $1$. We define a specific class of M-digraphs with the third sets consisting of finitely many vertices of the graphs. The third sets here are possibly empty.
\begin{Def}
A {\it simple} Morse digraph or an {\it SM-digraph} $(G,c_G,G_{\rm m})$ is a triplet with $(G,c_G)$ being an M-graph and $G_{\rm m}$ being a set of finitely many vertices of sources possibly empty such that the degrees of vertices of $G$ are all $1 $, $2$ or $3$, that at distinct vertices of the graph outside $G_{\rm m}$, the values of $c_G$ are mutually distinct and that in the case where $G_{\rm m}$ is non-empty the following hold.
\begin{itemize}
	\item The values of $c_G$ outside the non-empty set $G_{\rm m}$ are always greater than the minimum $m_{c_G}$ of $c_G$
	\item The level set ${c_G}^{-1}(m_{c_G})$ is $G_{\rm m}$. 
\end{itemize}
\end{Def}
A {\it simple} Morse function $f:M \rightarrow \mathbb{R}$ is a Morse function which has no critical point on the boundary and the values at distinct critical points of which are always distinct. In addition, for the definition, in the case the boundary $\partial M$ is non-empty, we pose a condition that for the minimum $m_f$ the level set $f^{-1}(m_f)$ coincides with the boundary $\partial M$ of $M$. 

For related theory on these Morse functions and graphs, consult \cite{gelbukh, michalak} as recent studies and see also \cite{kitazawa6}. Note that the notion of SM-digraph is formulated first in the present paper as the author knows. We also define {\it isomorphisms} for the triplets, as isomorphisms of digraphs mapping the third set of a triplet onto the third set of another triplet.

Theorem \ref{thm:1} is regarded as an important fact obtained by summarizing several results on structures of (simple) Morse functions on compact surfaces and compact manifolds of dimensions at least $3$. For this, see also \cite{kulinich, martinezalfaromezasarmientooliveira} for example.
\begin{Thm}
\label{thm:1}
\begin{enumerate}
\item \label{thm:1.1} For a simple Morse function $f:M \rightarrow \mathbb{R}$ on a compact and connected manifold of dimension at least $2$, the Reeb digraph $(R_f,\bar{f})$ is regarded as an SM-graph $(R_f,\bar{f},{R_f}_{\rm m})$ with the following rule. 
\begin{itemize}
\item The third set ${R_f}_{\rm m}$ is empty if $M$ is closed.
	\item  The third set ${R_f}_{\rm m}$ is the level set ${\bar{f}}^{-1}(m_f)$ for the minimum $m_f$ of $f$ if the boundary $\partial M$ of $M$ is non-empty.\end{itemize}
\item  \label{thm:1.2}
For a simple Morse function $f:M \rightarrow \mathbb{R}$ on a compact, connected, and orientable surface $M$,  the Reeb digraph $(R_f,\bar{f})$ is regarded as an SM-graph $(R_f,\bar{f},{R_f}_{\rm m})$ such that the degrees of vertices of the graph are $1$ or $3$, under the rule above.
\item \label{thm:1.3} From an SM-digraph $(G,c_G,G_{\rm m})$ such that the degrees of vertices of the graph are $1$ or $3$, we have a simple Morse function $f:M \rightarrow \mathbb{R}$ on a compact, connected, and orientable surface $M$ with $\bar{f}=c_G$
and ${R_f}_{\rm m}=G_{\rm m}$, under the rule for ${R_f}_{\rm m}$ above, and up to $\mathcal{R}$-equivalence and isomorphisms of the triplets.
\end{enumerate}
\end{Thm}
The following is an important SM-digraph.
\begin{Def}
	\label{def:2}
	Let $(G,c_G,G_{\rm m})$ be an SM-digraph with $G_{\rm m}$ being non-empty and with $m_{c_G}$ being the minimum of $c_G$. We can have an M-digraph $(G^{\ast},{c_G}^{\ast})$ admitting an isomorphism ${\phi}_G$ from the graph $G$ onto $G^{\ast}$ with the relation $m_{c_G}-{{c_G}^{\ast}}({\phi}_G(p))=c_G(p)-m_{c_G}$ 	for each point $p \in G$: note that the orientation is reversed. By identifying $G_{\rm m}$ and ${\phi}_G(G_{\rm m})$ by the isomorphism ${\phi}_G$, we have an SM-digraph whose vertex set is $(G \sqcup G^{\ast})-(G_{\rm m} \bigcup {\phi}_G(G_{\rm m}))$. Furthermore, set the function on the graph as the continuous function canonically defined from $c_G$ and ${c_G}^{\ast}$ and the third set as the empty set. This is the {\it double} of $(G,c_G,G_{\rm m})$.
	\end{Def}
\subsection{Classifications of round fold maps by Saeki with the author.}
We summarize important facts and arguments from \cite{kitazawasaeki}, through Theorems \ref{thm:2}, \ref{thm:3}, and \ref{thm:4}.
\begin{Thm}[\cite{kitazawasaeki}]
\label{thm:2}
In the case the dimension $m \geq 5$, a round fold map on an $m$-dimensional closed, connected and orientable manifold is uniquely determined from the $\mathcal{R}$-equivalence class of its page function, up to $\mathcal{A}$-equivalence. On the other hand, from a Morse function on a compact, connected and orientable surface, we have a round fold map whose page function is $\mathcal{R}$-equivalent to it, uniquely up to $\mathcal{A}$-equivalence.  
\end{Thm}
\begin{Thm}[\cite{kitazawasaeki}]
\label{thm:3}
Let $m \geq 4$ be an integer. For an $m$-dimensional closed, connected and orientable manifold $M$, the following hold.
\begin{enumerate}

\item \label{thm:3.1} Let  $l \geq 0$ be an integer. The $m$-dimensional manifold $M$ admits a round fold map $f:M \rightarrow {\mathbb{R}}^{m-1}$ with $f(M)$ being homeomorphic to $D^{m-1}$ and with the Reeb graph of its page function being of 1st Betti number $l$ if and only if $M$ is diffeomorphic to a connected sum ${\sharp}_{j=1}^l (S^{1} \times S^{m-1})$.
\item \label{thm:3.2} Let $m>4$ and in this case, the $m$-dimensional manifold $M$ admits a round fold map $f:M \rightarrow {\mathbb{R}}^{m-1}$ with $f(M)$ being homeomorphic to $S^{m-2} \times D^1$ and with the Reeb graph of its page function being of 1st Betti number $l$ if and only if $M$ is diffeomorphic to $S^{m-1} \times S$ with $S$ being a closed, connected and orientable surface of genus $l$.
\item \label{thm:3.3} Let $m=4$ and in this case, for the $m$-dimensional manifold, $M$ admits a round fold map $f:M \rightarrow {\mathbb{R}}^{m-1}$ with $f(M)$ being homeomorphic to $S^{m-2} \times D^1$ and with the Reeb graph of its page function being of 1st Betti number $l$ and having at least $2$ edges if and only if $M$ is diffeomorphic to $S^{m-1} \times S$ with $S$ being a closed, connected and orientable surface of genus $l$.
\item \label{thm:3.4} Let $m=4$ and in this case, for the $m$-dimensional manifold, $M$ admits a round fold map $f:M \rightarrow {\mathbb{R}}^{m-1}$ with $f(M)$ being homeomorphic to $S^{m-2} \times D^1$ and with the Reeb graph of its page function having exactly one edge if and only if $M$ is diffeomorphic to either $S^2 \times S^2$ or a non-trivial smooth bundle over $S^2$ whose fiber is diffeomorphic to $S^2$.
\end{enumerate}
\end{Thm}
A round fold map $f:M \rightarrow {\mathbb{R}}^{m-1}$ in Theorem \ref{thm:2} is represented as follows. This round fold map is the product map of a page function $\tilde{f_{{\rm P},x_0}}$ and the identity map on $S^{m-2}$ up to $\mathcal{A}$-equivalence in the case where $f(M)$ is homeomorphic to $S^{m-2} \times D^1$. In the case where $f(M)$ is homeomorphic to $D^{m-1}$, the round fold map is, up to $\mathcal{A}$-equivalence, a smooth map obtained by gluing the product map of a page function $\tilde{f_{{\rm P},x_0}}$ and the identity map on $S^{m-2}$ and the product bundle $D^{m-1} \times C$ with $C$ being non-empty and a disjoint union of copies of $S^1$ along the boundaries by the product map of the identity map on the base space $S^{m-2}$ and that on the fiber $C$, where the identifications are chosen naturally for these manifolds of base spaces and fibers.

This construction is of course applied for constructing a round fold map seen as a simplest one into the Euclidean space of an arbitrary dimension $n \geq 2$ whose page function is a prescribed Morse function with its critical value set being scaled suitably on an ($m-n$)-dimensional compact manifold with no boundary or with its boundary being the preimage of its minimum. The resulting round fold map is of course unique up to $\mathcal{A}$-equivalence. This construction is called {\it trivial spinning construction}. We use this name, respecting \cite{kitazawasaeki} with \cite{kitazawa2}, for example. A round fold map constructed in such a way is a {\it trivial-spinning-constructed} map or a {\it TSC} map.
We can improve several theorems presented above.
\begin{Thm}
\label{thm:4}
In Theorem \ref{thm:2} and in Theorem \ref{thm:3} {\rm (}\ref{thm:3.1}, \ref{thm:3.2}, \ref{thm:3.3}{\rm )}, each round fold map is a TSC map. In Theorem \ref{thm:3} {\rm (}\ref{thm:3.4}{\rm )}, the manifold of the domain of a TSC map is diffeomorphic to $S^2 \times S^2$. Together with Theorem \ref{thm:1} {\rm (}\ref{thm:1.2}{\rm )}, these TSC maps in the present situation and the SM-digraphs $(G,c_G,G_{\rm m})$ whose vertices are always of degree $1$ or $3$ correspond one-to-one, where $(G,c_G,G_{\rm m})$ is for a page function of TSC map. 
\end{Thm}
Note that for round fold maps $f:M \rightarrow {\mathbb{R}}^{m-1}$ on $m$-dimensional closed, connected and (possibly) non-orientable manifolds with $m \geq 4$, similar theorems have been shown in \cite{kitazawasaeki}. However we omit related exposition, for simplicity.
\subsection{Special generic maps in the real algebraic category.}
We review a main ingredient of \cite{kitazawa5} in a self-contained way. For real algebraic arguments in the present paper, see \cite{kollar} and see also the textbook \cite{bochnakcosteroy}.
\begin{Thm}
\label{thm:5}
Let $D \subset {\mathbb{R}}^n$ be a non-empty open set of ${\mathbb{R}}^n$ such that for the closure $\overline{D}$ considered there, $\overline{D}-D$ is the zero set of a real polynomial function $f_{D,S}$, that the rank of the differential of $f_{D,S}$ there is always $1$, and that $D=\{x \in {\mathbb{R}}^n \mid \ f_{D,S}(x)>0\}$. In this situation, the zero set $M_{D,S}:=\{(x,{(y_j)}_{j=1}^{m-n+1}) \in {\mathbb{R}}^n \times {\mathbb{R}}^{m-n+1} \mid f_{D,S}(x)-{\Sigma}_{j=1}^{m-n+1} {y_j}^2=0 \}$ with $m$ being an integer satisfying $m \geq n$ is an $m$-dimensional smooth submanifold with no boundary of ${\mathbb{R}}^{m+1}$. Furthermore, the restriction of ${\pi}_{m+1,n}$ there is a special generic map onto  $\overline{D} \subset {\mathbb{R}}^n$ whose singular set is $\{(x,0) \in ({\mathbb{R}}^n \times {\mathbb{R}}^{m-n+1}) \bigcap M_{D,S}\}$.
\end{Thm}
\begin{proof}
By implicit function theorem, we prove that the zero set $M_{D,s}$ is a smooth submanifold of ${\mathbb{R}}^{m+1}$. In the case $\{(x_0,y_0) \in (D \times {\mathbb{R}}^{m-n+1}) \bigcap M_{D,S}\}$, the value of the partial derivative of the function $f_{D,S}(x)-{\Sigma}_{j=1}^{m-n+1} {y_j}^2$ by some variable $y_j$ is not $0$ at the point $(x_0,y_0)$. In the case $\{(x_0,y_0) \in ((\overline{D}-D) \times {\mathbb{R}}^{m-n+1}) \bigcap M_{D,S}\}$, the value of the partial derivative of the function $f_{D,S}(x)-{\Sigma}_{j=1}^{m-n+1} {y_j}^2$ by some variable $x_i$ from the coordinate $x=(x_1,\cdots x_n)$ is not $0$ at the point $(x_0,y_0)=(x_0,0)$. From this, we can apply implicit function theorem to show that $M_{D,S}$ is an $m$-dimensional smooth submanifold with no boundary of ${\mathbb{R}}^{m+1}$.

By fundamental general arguments on the zero sets of real polynomial maps and local coordinates around $\{(x,0) \in ({\mathbb{R}}^n \times {\mathbb{R}}^{m-n+1}) \bigcap M_{D,S}\}$, the restriction of ${\pi}_{m+1,n}$ there is shown to be a special generic map whose singular set is $\{(x,0) \in ({\mathbb{R}}^n \times {\mathbb{R}}^{m-n+1}) \bigcap M_{D,S}\}$. For the arguments on the local structure of the manifolds and maps here, remember implicit function theorem. Related to this, see also \cite[Discussion 14, Claim 14.2]{kollar}, for example.
This completes the proof.
\end{proof}
\subsection{Approximation of smooth functions and maps by real polynomial functions and maps.}

In singularity theory, the {\it $C^k$ topology} on the space of all smooth maps between two smooth manifolds is an important topology, where $k$ is a non-negative integer or $\infty$ and for the class of smoothness of maps between manifolds.
We omit rigorous definitions and arguments related to such a topology. Consult \cite{golubitskyguillemin} again for this.
This topology can be naturally defined in the following way, roughly speaking: two smooth maps are close if their values at each point on each compact set and their $j$-th derivatives at each point on the compact set are close for $1 \leq j \leq k$. The {\it $C^{\infty}$ topology} is (defined by) the union of (all open sets for) $C^k$ topologies for all $k \geq 0$. 
The {\it Whitney $C^k$ topology} with $k=\infty$ is shortly presented as a topology on a group of diffeomorphisms, before. This is a stronger variant, where "each compact set" is replaced by the manifold of the domain: in the case the manifold of the domain is compact, the topology is same as the $C^k$ topology.

It is important in our paper that, as surveyed in \cite[Proposition 1.3.7]{lellis}, an article on history and advanced studies of real algebraic geometry, for example, smooth maps between Euclidean spaces are approximated by real polynomial maps ({\it real algebraic} maps) on a compact set, in the $C^k$ topology, for any non-negative given integer $k$.
\subsection{Algebraic domains and Poicar\'e-Reeb graphs of them and related theory established by Bodin, Popescu-Pampu, and Sorea.}
We review \cite{bodinpopescupampusorea} with additional original exposition by the author.

An {\it algebraic domain} $D$ in ${\mathbb{R}}^2$ is a bounded non-empty open set there surrounded by mutually disjoint real algebraic curves. Here, a {\it real algebraic curve} $C$ means a connected component of the zero set of a real polynomial function $f_C$ such that the rank of the differential of $f_C$ on $C$ is always $1$: in short $C$ is {\it non-singular}.
We can consider the closure $\overline{D}$ of $D$ in ${\mathbb{R}}^2$ and the restriction of ${\pi}_{2,1}$ there. We can define the quotient space for the restriction ${\pi}_{2,1} {\mid}_{\overline{D}}$ as in the Reeb graph. We can give the quotient space the structure of a graph in the following way. 
\begin{itemize}
\item Each element of the quotient space represents a contour of ${\pi}_{2,1} {\mid}_{\overline{D}}$.

\item  An element of the quotient space is its vertex if and only if it represents a contour containing a critical point of the restriction ${\pi}_{2,1} {\mid}_{\overline{D}-D}$, where $\overline{D}-D$ is the disjoint union of finitely many copies of $S^1$ smoothly embedded in ${\mathbb{R}}^2$ as real algebraic curves disjointly.
\end{itemize}
\begin{Def}
\label{def:3}
The graph above is the {\it Poincar\'e-Reeb graph} of $D$ and we can also regard this as a digraph naturally (the {\it Poincar\'e-Reeb digraph} of $D$).
\end{Def}
\begin{Thm}[\cite{bodinpopescupampusorea}]
\label{thm:6}
Consider an SM-digraph $(G,c_G,G_{\rm m})$ with $G_{\rm m}$ being empty and with each vertex being of degree $1$ or $3$ which is represented as the composition of an embedding $\tilde{c_G}:G \rightarrow {\mathbb{R}}^2$ with ${\pi}_{2,1}$. Then we can choose an algebraic domain $D_{G,c_G} \subset {\mathbb{R}}^2$ whose Poincar\'e-Reeb digraph is isomorphic to $(G,c_G)$.
\end{Thm}
We review the original proof, presented in the third section of the original paper with this statement. Some arguments may be different from the original ones. However, they are almost essentially same. 
In short, ${D_{G,c_G}}^{\prime} \subset {\mathbb{R}}^2$ is chosen as a small connected open neighborhood of the graph $\tilde{c_G}(G)$ surrounded by disjoint union of finitely many connected smooth curves ${C_j}^{\prime}$. We see each connected curve as a connected component of the zero set of a smooth function $f_{{C_j}^{\prime}}$ the rank of the differential of which is always $1$ on ${C_j}^{\prime}$, whose values are non-zero on the set ${D_j}^{\prime}-{C_j}^{\prime}$ in an open connected region ${D_j}^{\prime}$ in ${\mathbb{R}}^2$ containing the closure of ${D_{G,c_G}}^{\prime}$ in ${\mathbb{R}}^2$, and whose values are always $0$ on ${\mathbb{R}}^2-{D_j}^{\prime}$. By applying technique of approximation before, ${C_j}^{\prime}$ and $f_{{C_j}^{\prime}}$ are changed to a real algebraic curve $C_j$ and real polynomial function $f_{C_j}$.

What follows is not explained explicitly in the original article. Similar exposition is also shortly in the preprint \cite{kitazawa7} of the author, where we do not need knowledge on the preprint.

The curves ${C_j}^{\prime}$ are chosen in such a way that the restriction of ${\pi}_{2,1}$ to their union is a simple Morse function. As presented in \cite[Lemma 3.6]{bodinpopescupampusorea} and its proof, $C_j$ and $f_{C_j}$ are obtained by the technique of approximation in the $C^2$ topology. There, Bodin, Popescu-Pampu, and Sorea discuss which $1$-st and $2$-nd derivatives we should approximate and they approximate suitable derivatives. What follows is exposition respecting singularity theory of differentiable maps more. From fundamental arguments on singularity theory such as $C^k$ topologies, Whitney $C^k$ topologies, and local stability of smooth maps, $C_j$ and $f_{C_j}$ are obtained by the technique of approximation in the $C^2$ topology in such a way that the restriction of ${\pi}_{2,1}$ to the union of the curves $C_j$ is a simple Morse function. We can omit precise arguments on the choice and approximation of the derivatives. On the other hand, the original exposition is of course true and we can follow the exposition without fundamental knowledge on singularity theory of differentiable maps. Furthermore, as presented near \cite[Special Case 5]{kollar}, by adding non-positive or non-negative polynomial functions suitably, $C_j$ can be regarded as the zero set of $f_{C_j}$. Summarizing, we have the following.
\begin{Thm}
\label{thm:7}
In Theorem \ref{thm:6}, we can have the desired case with the following.
\begin{enumerate}
\item \label{thm:7.1} Each curve $C_j$ from all real algebraic curves surrounding $D_{G,c_G} \subset {\mathbb{R}}^2$ is the zero set of some corresponding real polynomial function $f_{C_j}$. 
\item \label{thm:7.2} The restriction ${\pi}_{2,1} {\mid}_{{\sqcup}_{j} C_j}$ is a simple Morse function.
\end{enumerate}
\end{Thm}
 \section{Main Theorem \ref{mthm:0} revisited.}
We present Main Theorems \ref{mthm:0} again, more precisely.

 Main Theorem \ref{mthm:1} is for Main Theorem \ref{mthm:0} (\ref{mthm:0.1}).
\begin{MainThm}
\label{mthm:1}
In Main Theorem \ref{mthm:0} {\rm (}\ref{mthm:0.1}{\rm )}, the zero set $M$ is of the form
$$\{x:=(x_1,\cdots x_{m-1},x_m,x_{m+1}) \in {\mathbb{R}}^{m+1} \mid {\prod}_{i \in I} h_i({\Sigma}_{j=1}^{m-1} {x_j}^2,x_{m})-{x_{m+1}}^2=0\}$$
with some family $\{h_i:{\mathbb{R}}^2 \rightarrow \mathbb{R}\}_{i \in I}$ of finitely many real polynomial functions, indexed by $i \in I$ for some finite set $I$.
\end{MainThm}
\begin{proof}
For a given round fold map, we consider its page function $\tilde{f_{{\rm P},x_0}}$, where the function $\bar{\tilde{f_{{\rm P},x_0}}}$ is the composition of an embedding into ${\mathbb{R}}^2$ with ${\pi}_{2,1}$. The function $\tilde{f_{{\rm P},x_0}}$ is seen as a simple Morse function on a closed, connected and orientable surface diffeomorphic to $S$ into $\{t \geq \frac{1}{2}\} \subset \mathbb{R}$. By considering Theorems \ref{thm:6} and \ref{thm:7}, we have an algebraic domain $D_{R_{\tilde{f_{{\rm P},x_0}}},\bar{\tilde{f_{{\rm P},x_0}}}}$ whose Poincar\'e-Reeb digraph is isomorphic to $(R_{\tilde{f_{{\rm P},x_0}}},\bar{\tilde{f_{{\rm P},x_0}}})$ and have each real polynomial function "$f_{C_i}$ of Theorem \ref{thm:7} (, with $j$ being changed into $i$,)" for the present situation, denoted by $h_i$. By using $h_i({\Sigma}_{j=1}^{m-1} {x_j}^2,x_{m})$ and applying Theorem \ref{thm:5} with the sign of each $h_j$ being suitably chosen and $f_{D,S}(x)={\prod}_{i \in I} h_i({\Sigma}_{j=1}^{m-1} {x_j}^2,x_{m})$ with $n=m$ under the notation from Theorem \ref{thm:5}, we have a special generic map on the zero set $M$ into ${\mathbb{R}}^m$ as ${\pi}_{m+1,m} {\mid}_{M}$. The restriction of ${\pi}_{m+1,m-1}$ to $M$ is also shown to be a round fold map by fundamental observations on the singularity, respecting \cite[Propsition 6.1 and Corollary 6.2]{saekisuzuoka}. For singularity of these fold maps, see also \cite{fukuda} for example. Due to Theorem \ref{thm:3} (\ref{thm:3.2}, \ref{thm:3.3}) and Theorem \ref{thm:4}, the round fold map is a desired round fold map (up to $\mathcal{A}$-equivalence).   

This completes the proof.
\end{proof}
 Main Theorem \ref{mthm:2} is for Main Theorem \ref{mthm:0} (\ref{mthm:0.2}).
\begin{MainThm}
\label{mthm:2}
In Main Theorem \ref{mthm:0} {\rm (}\ref{mthm:0.2}{\rm )}, the zero set $M$ is of the form
$$\{x:=(x_1,\cdots x_{m-1},x_m,x_{m+1}) \in {\mathbb{R}}^{m+1} \mid {\prod}_{i \in I} (h_i({\Sigma}_{j=1}^{m-1} {x_j}^2,x_{m})) \times H_{\rm S}((x_j)_{j=1}^m)-{x_{m+1}}^2=0\}$$
with a finite family $\{h_i:{\mathbb{R}}^2 \rightarrow \mathbb{R}\}_{i \in I}$ of suitable real polynomial functions and a real polynomial function $H_{\rm S}:{\mathbb{R}}^m \rightarrow \mathbb{R}$ represented as the product of finitely many real polynomials $h_{{\rm S},i^{\prime}}$ whose zero sets $\{x:=(x_1,\cdots x_{m}) \in {\mathbb{R}}^{m} \mid  h_{{\rm S},i^{\prime}}(x)=0\}$ are submanifolds of ${\mathbb{R}}^m$ diffeomorphic to $S^{m-1}$. 
\end{MainThm}
\begin{proof}
The story of the proof is essentially same as that of Main Theorem \ref{mthm:1}, where arguments on an algebraic domain with curves for Theorem \ref{thm:6} and \ref{thm:7} are a bit different.

For a given round fold map, we consider its page function $\tilde{f_{{\rm P},x_0}}$, where the function $\bar{\tilde{f_{{\rm P},x_0}}}$ is the composition of an embedding into ${\mathbb{R}}^2$ with ${\pi}_{2,1}$. The function $\tilde{f_{{\rm P},x_0}}$ is seen as a simple Morse function on a compact, connected and orientable surface diffeomorphic to $S$ whose boundary $\partial S$ is non-empty into $\{t \geq \frac{1}{2}\} \subset \mathbb{R}$.
We replace the minimum of $\bar{\tilde{f_{{\rm P},x_0}}}$ by $0$ and having an SM-digraph $(R_{\tilde{f_{{\rm P},x_0,0}}},\bar{\tilde{f_{{\rm P},x_0,0}}},{R_{\tilde{f_{{\rm P},x_0,0}}}}_{\rm m})$ isomorphic to the original SM-digraph. We consider its double $(D({R_{\tilde{f_{{\rm P},x_0,0}}}}),D(\bar{\tilde{f_{{\rm P},x_0,0}}}))$.

We apply Theorems \ref{thm:6} and \ref{thm:7} for $(D({R_{\tilde{f_{{\rm P},x_0,0}}}}),D(\bar{\tilde{f_{{\rm P},x_0,0}}}))$. First, we consider an embedding of $D({R_{\tilde{f_{{\rm P},x_0,0}}}})$ into ${\mathbb{R}}^2$ with respect to the axis $\{(0,t) \mid t \in \mathbb{R}\}$ such that its composition with ${\pi}_{2,1}$ is $D(\bar{\tilde{f_{{\rm P},x_0,0}}})$ and obtain a domain $D_{D({R_{\tilde{f_{{\rm P},x_0,0}}}}),D(\bar{\tilde{f_{{\rm P},x_0,0}}})}$ which is "almost" an algebraic domain, where some curves $C_j$ are chosen to be only smooth. The philosophy of the method is similar to that in a similar step in the proof of Main Theorem \ref{mthm:1}.
We can have this with the following properties thanks to arguments around Theorems \ref{thm:6} and \ref{thm:7}. We abuse the notation from Theorems \ref{thm:6} and \ref{thm:7} and discussions around them.
\begin{itemize}
\item Each curve $C_j:=C_{j_1}$ in $\{s>0\} \times \mathbb{R} \subset {\mathbb{R}}^2$ is a real algebraic curve.
\item Each curve $C_j:=C_{j_2}$ in $\{s<0\} \times \mathbb{R} \subset {\mathbb{R}}^2$ is a curve obtained by a reflection of the uniquely chosen curve $C_{j_1}$ in $\{s>0\} \times \mathbb{R} \subset {\mathbb{R}}^2$ above, along the axis $\{(0,t) \mid t \in \mathbb{R}\}$. 
\item Each curve $C_j:=C_{j_3}$ whose intersection with the axis $\{(0,t) \mid t \in \mathbb{R}\}$ is non-empty is, by the reflection along the axis $\{(0,t) \mid t \in \mathbb{R}\}$, transformed onto itself. 
\end{itemize}
In short, the domain $D_{D({R_{\tilde{f_{{\rm P},x_0,0}}}}),D(\bar{\tilde{f_{{\rm P},x_0,0}}})}$ and the curves are located with symmetry with respect to the axis $\{(0,t) \mid t \in \mathbb{R}\}$.

For $C_{j_1}$ and $C_{j_2}$ above, we can define and consider $h_i({\Sigma}_{j=1}^{m-1} {x_j}^2,x_{m})$ as in Main Theorem \ref{mthm:1}. For $C_{j_3}$ above, we define smooth functions ${h_{{\rm S}^{\prime},i^{\prime}}}:{\mathbb{R}}^m \rightarrow \mathbb{R}$, before we obtain real polynomial functions $h_{{\rm S},i^{\prime}}$ whose zero sets are submanifolds of ${\mathbb{R}}^m$ diffeomorphic to $S^{m-1}$. As in the last of Remark \ref{rem:1}, presented later, by fundamental arguments on approximation in singularity theory of smooth maps, we can approximate ${h_{{\rm S}^{\prime},i^{\prime}}}:{\mathbb{R}}^m \rightarrow \mathbb{R}$ by a real polynomial function $h_{{\rm S},i^{\prime}}$ and by Theorem \ref{thm:5}, the arguments on singularity of special generic maps and fold maps presented in the end of the proof of Main Theorem \ref{mthm:1}, and Theorem \ref{thm:3} (\ref{thm:3.1}) with Theorem \ref{thm:4}, we have a desired round fold map (up to $\mathcal{A}$-equivalence).

This completes the proof. 
\end{proof}
We close the present paper by remarks on real algebraic realizations of smooth maps from the general viewpoint and our real algebraic realization.  
\begin{Rem}
\label{rem:1}
A smooth map $f:M \rightarrow {\mathbb{R}}^n$ on an $m$-dimensional manifold $M$ with no boundary is, by fundamental arguments on differential topology, represented as the composition of a smooth embedding into the Euclidean space ${\mathbb{R}}^{m+m^{\prime}}$ of sufficiently high dimension $m+m^{\prime}$ with $m^{\prime}>0$ being an integer, with ${\pi}_{m+m^{\prime},n}$. Let $M$ be closed. In such a case, we can approximate the embedding in the $C^k$ topology (the Whitney $C^k$ topology) and the manifold $M \subset {\mathbb{R}}^{m+m^{\prime}}$ is seen as a smooth compact submanifold with no boundary and the zero set of some real polynomial map $e_{m+m^{\prime},m^{\prime}}:{\mathbb{R}}^{m+m^{\prime}}  \rightarrow {\mathbb{R}}^{m^{\prime}}$. For this, consult \cite{lellis}, with \cite{kollar}, again. This is also related to Nash and Tognoli's theory \cite{nash, tognoli} on realizing a smooth closed manifold as the zero set of a real polynomial map in such a way that the differential of the restriction of the real polynomial map $e_{m+m^{\prime},m^{\prime}}$ there is always of the full rank. For real algebraic geometry, see also the textbook \cite{bochnakcosteroy}, for example.   

In the case $f$ is a so-called {\it generic} smooth map on a closed manifold such as a fold map, we can have a real algebraic realization of a generic smooth map sufficiently close to the original map and similar from the viewpoint of singularity. More strongly, if a generic smooth map is a so-called {\it stable} map, we have a real algebraic realization of the original map up to $\mathcal{A}$-equivalence. Round fold maps are stable maps and related arguments are used in the proof of Main Theorem \ref{mthm:2}.
\end{Rem}
\begin{Rem}
\label{rem:2}
Our realization is more explicit than the case of Remark \ref{rem:1}.
\end{Rem}

Last, we give an improvement of Main Theorem \ref{mthm:2} in a specific case.
\begin{MainThm}
\label{mthm:3}
In Main Theorem \ref{mthm:0}, let $F_0(f)$ be connected. 
Then we have a case of Main Theorem \ref{mthm:2} with $h_{{\rm S},i^{\prime}}(x)$ having the form $\pm ({\Sigma}_{j=1}^{m} ({x_j}^2)-r_{i^{\prime}})$ with $r_{i^{\prime}}>0$. 
\end{MainThm}
\begin{proof}
The function $\bar{\tilde{f_{{\rm P},x_0}}}$ is the composition of an embedding into ${\mathbb{R}}^2$ with ${\pi}_{2,1}$. In the proof of Main Theorem \ref{mthm:2}, the embedding of the graph into ${\mathbb{R}}^2$ can chosen satisfying the symmetry with respect to $\{(0,t) \mid t \in \mathbb{R}\}$, naturally. We consider the complementary set of the image of the embedded graph in ${\mathbb{R}}^2$. This is the disjoint union of finitely many connected open sets of course. The connected open sets intersecting $\{(0,t) \mid t \in \mathbb{R}\}$ satisfy the following.
\begin{itemize}
\item Exactly one set $R_0$ is unbounded and the set $\overline{R_0}-R_0$ is the union of two curves $C_{R_0,1}$ and $C_{R_0,2}$, where the set $\overline{R_0}$ is the closure of $R_0$ in ${\mathbb{R}}^2$. The restrictions of ${\pi}_{2,1}$ to the two curves are both injective and $C_{R_0,1} \bigcap C_{R_0,2}$ is homeomorphic to the disjoint union of two copies of $D^1$. 
\item The remaining sets $R_{a}$ ($a \neq 0$) are all bounded. Each set $\overline{R_{a}}-R_a$ ($a \neq 0$) is the union of two curves $C_{R_a,1}$ and $C_{R_a,2}$, where the set $\overline{R_a}$ is the closure of $R_a$ in ${\mathbb{R}}^2$. The restrictions of ${\pi}_{2,1}$ to the two curves are both injective and $C_{R_a,1} \bigcap C_{R_a,2}$ is a discrete two-point set consisting of two vertices of degree $3$ of the graph. 
\end{itemize}
In knowing this, it is important to remember that the graph is connected and that the number of vertices of degree $1$ of the graph is exactly $2$. The latter fact comes from the assumption that $F_0(f)$ is connected. Furthermore, we can deform the embedded graph vertically, suitably, and have the following. For the second case $R_{a}$ ($a \neq 0$), $C_{R_a,1}$ is below the set $C_{R_a,1} \bigcap C_{R_a,2}$, $C_{R_a,2}$ is above the set $C_{R_a,1} \bigcap C_{R_a,2}$, and we can have each curve $C_{j_3}=\{(x_1,x_{m+1}) \in {\mathbb{R}}^2 \mid {x_1}^2+{x_{m+1}}^2-r_{j_3}=0\}$ with a suitable $r_{j_3}>0$ in the proof of Main Theorem \ref{mthm:2}. From this, we have a desired real polynomial function $h_{{\rm S},i^{\prime}}(x)$. This completes the proof.
\end{proof}

\end{document}